\documentclass[oneside,a4paper,12pt]{amsart} 
\usepackage{amssymb}
\usepackage{amscd}
\usepackage[all,cmtip]{xy}
\usepackage{verbatim}
\usepackage{lscape}
\usepackage{enumerate}
\usepackage{xcolor}
\theoremstyle{plain}
\newtheorem{theorem}{Theorem}[section]
\newtheorem{lemma}[theorem]{Lemma}

\newtheorem{corollary}[theorem]{Corollary}
\newtheorem{proposition}[theorem]{Proposition}

\theoremstyle{definition}

\newtheorem{definition}[theorem]{Definition}

\theoremstyle{remark}
\newtheorem*{remark}{Remark}

\newtheorem*{example}{Example}
 
\DeclareMathOperator{\Ker}{Ker}

 \DeclareMathOperator{\Hom}{Hom}

\DeclareMathOperator{\Ima}{Im} \DeclareMathOperator{\res}{res}

\DeclareMathOperator{\reg}{reg}

\DeclareMathOperator{\rank}{rank}

\DeclareMathOperator{\Gl}{Gl}

\begin{document}

\title{Rank, Coclass and Cohomology}

\author{Peter Symonds}
\thanks{Partially supported by an International Academic Fellowship from the Leverhulme Trust.}
\address{School of Mathematics\\
         University of Manchester\\
     Manchester M13 9PL\\
     United Kingdom}
\email{Peter.Symonds@manchester.ac.uk}


\begin{abstract}
We prove that for any prime $p$ the finite $p$-groups of fixed coclass have only finitely many different mod-$p$ cohomology rings between them. This was conjectured by Carlson; we prove it by first proving a stronger version for groups of fixed rank.
\end{abstract}

\maketitle

\section{Introduction}
\label{sec:intro}

Jon Carlson in \cite[{\S}1]{carlson} conjectured that the $p$-groups of given coclass have only finitely many isomorphism classes of mod-$p$ cohomology rings between them and he presented a proof for $p=2$. In this paper we prove Carlson's conjecture for all primes. In fact, we prove a stronger conjecture for $p$-groups of given rank, due to D\'{i}az, Garaialde and Gonz\'{a}lez \cite[5.2]{dgg2}. The rank of a $p$-group is the smallest number $r$ such that any subgroup can be generated by $r$ elements.

\begin{theorem}
	\label{th:main}
	For fixed $p$ and $r$, the $p$-groups of rank $r$ have only finitely many graded isomorphism classes of cohomology rings between them.
\end{theorem}

For a given prime $p$, the rank of a $p$-group is bounded in terms of the coclass (see Proposition~\ref{pr:rank}), so this implies the truth of Carlson's Conjecture.	

The coclass of a $p$-group is $n-c$, where $p^n$ is the order of the group and $c$ is the length of its lower central series. 
The coclass classification of finite $p$-groups was developed by Leedham-Green, Shalev and others (see \cite{LG}). It provides a great deal of information about the structure of $p$-groups of given coclass; however, much more precise descriptions are believed to be possible. One motivating conjecture \cite[{\S}1]{def} claims that the $p$-groups of a given coclass can be divided up into finitely many families, called coclass families, in such a way that each coclass family can be described by a single parametrised presentation; in addition, many structural invariants, such as cohomology, Schur multipliers and automorphism groups, can also be given by a single parametrised presentation on each coclass family. Indeed, Eick and Green \cite[1.3]{eg} conjecture that cohomology is constant on a coclass family provided the group is big enough. Our result can be viewed as a strong corroboration of this conjecture, in the sense that it shows that the conjecture can always be made true by refining the coclass families if necessary. 

The simplest example is that of the 2-groups of coclass 1. The three coclass families are dihedral, semi-dihedral and generalised quaternion. The cohomology ring is constant on each family, provided the group has order at least 16, even though there is no homomorphism between different groups in the same family that induces a cohomology isomorphism.

  Ideally, one would like to show that groups with similar structure have isomorphic cohomology rings: that is not what we or Carlson do. Instead, we use Benson's Regularity Conjecture, as proved in \cite{Sy1}, to bound the degrees of the generators and relations of the cohomology rings and hence the number of isomorphism classes.

A different proof for $p=2$ was developed recently by D\'{i}az, Garaialde and Gonz\'{a}lez \cite[1.1]{dgg}, which provides a more structural explanation for the isomorphisms, even at odd primes, where the method yields partial information. 

Recently, Guralnick and Tiep \cite[1.3]{GT} asked whether it was possible to bound $\dim H^n(G;V)$, for $V$ a simple $kG$-module over an algebraically closed field $k$ of characteristic $p$ in terms only of $n$, the sectional $p$-rank of $G$ and, perhaps, $p$. They state that it is likely that the problem can be reduced to two parts, the case of simple groups and the case of $p$-groups. As a by-product of our methods we prove the second of these and we obtain a result that is essentially independent of the prime.

\begin{theorem}
	\label{th:gt}
	 If $G$ is a $p$-group of rank at most $r$ then $\dim H^i(G) \leq \binom{r(\lceil \log_2r \rceil  +3 +e)+i-1}{i}$, where $e=0$ for $p$ odd and $e=1$ for $p=2$. 
	 \end{theorem}

To put this in perspective, note that the bound is equal to $\dim H^i(E)$ when $E$ is an elementary abelian $p$-group of rank $r(\lceil \log_2r \rceil  +3 +e)$. Observe that this bound  depends only on $r$ and $i$; the prime $p$ does not appear. On the other hand, it grows much too fast with respect to $i$; we know from the work of Quillen \cite[7.8]{quillen} that as $i$ increases $\dim H^i(G)$ grows like $c\cdot i^{a-1}$, where $a$ is the maximum rank of an elementary abelian subgroup. This can be fixed at the expense of a bound that involves $p$.

\begin{theorem}
	\label{thm:gt2}
There is a function $X(p,r)$ such that if $G$ is a $p$-group of rank $r$ then $\dim H^i(G) \leq X(p,r) \cdot i^{a-1}$.
\end{theorem}

We can also extend these results to pro-$p$ groups.

\begin{corollary}
	\label{cor:pro-p} Theorems~\ref{th:main} and \ref{th:gt} also hold for pro-$p$ groups of bounded rank and the usual continuous cohomology.
	\end{corollary}

I wish to thank Antonio D\'{i}az Ramos for explaining to me some of his ideas on this subject and Bob Guralnick for pointing out the relevance of the results here to the question in \cite{GT}.

\section{Generators and Relations}
\label{sec:par}

We fix a prime $p$. All groups will be $p$-groups and all cohomology rings will have coefficients in the field of $p$ elements, $\mathbb F_p$, unless otherwise indicated. By an isomorphism of cohomology rings we mean an isomorphism as graded $\mathbb F_p$-algebras.

\begin{definition}
	A collection of parameters for the cohomology of a group $G$ is a finite collection of homogeneous elements $x_1,  \ldots , x_n \in H^*(G)$ such that $H^*(G)$ is finite over $\mathbb F_p[x_1,\ldots ,x_n]$. There is no requirement of minimality or independence and there may even be repetitions. The word \emph{collection} is meant to emphasise that this is not what is usually called a system of parameters.
	\end{definition}

Consider a set $C = \{ G_{\lambda} \}_{\lambda \in \Lambda}$ of finite $p$-groups.

\begin{proposition}
\label{pr:par}
The following conditions on $C$ are equivalent.
\begin{enumerate}
\item
The cohomology rings $H^*(G_{\lambda})$ fall into finitely many isomorphism classes.
\item
There exist numbers $D,M,N \in \mathbb N$ such that 
\begin{enumerate}
\item
each $H^*(G_{\lambda})$ has a collection of parameters $x_1,  \ldots , x_n$, with $n \leq N$ and $\deg x_i \leq D$, and
\item
 for each $\lambda$, $\dim_{\mathbb F_p} \oplus _{i=0}^L H^i(G_{\lambda}) \leq M$, where $L=\max \{ 2N(D-1), 1 \}$.
 \end{enumerate}
\end{enumerate}
\end{proposition}

\begin{proof}
That (1) implies (2) is trivial, so we concentrate on the other direction. By the Regularity Theorem \cite[0.2]{Sy1}, $\reg H^*(G_{\lambda}) =0$ and, as a consequence, $H^*(G_{\lambda})$ has generators as a graded-commutative $\mathbb F_p$-algebra in degrees at most $\max \{N(D-1),D \}$ and relations in degrees at most $\max \{2N(D-1), N(D-1)+1,D \}$, which is easily seen to be bounded by $L$, see \cite[2.1]{Sy2}.

Thus, each $H^*(G_{\lambda})$ is a quotient of a graded-commutative polynomial ring $\mathbb F_p [z_1, \ldots , z_\ell]$, with the $z_i$ homogeneous of degree at most $\max \{N(D-1),D \}$. We can assume the $z_i$ to be linearly independent over $\mathbb F_p$, so $\ell \leq M$.  The kernel of the map from $\mathbb F_p [z_1, \ldots , z_\ell]$ to $H^*(G_{\lambda})$ is an ideal generated in degrees at most $L$, hence is completely determined by its part in this range of degrees; there are only finitely many possibilities, simply because we are considering subsets of a finite set.
\end{proof}

\begin{lemma}
\label{le:dimbound}

Suppose that there are two sequences of numbers $U(i),V(i)$ such that each $G_{\lambda}$ has a normal subgroup $H_{\lambda}$ satisfying $\dim H^i(H_{\lambda}) \leq U(i)$ and $\dim H^i(G_{\lambda}/H_{\lambda}) \leq V(i)$. Then part (b) of condition (2) holds. If there is a number $W$ such that $|G_{\lambda}/H_{\lambda}| \leq p^W$ for all $\lambda$ then we can take $V(i)=\binom{W+i-1}{W-1}$.
\end{lemma}

\begin{proof}
Use the Lyndon-Hochschild-Serre spectral sequence for $H_\lambda \unlhd G_\lambda$. It is sufficient to bound the dimension of the entries on the $E_2$-page in the given range and this follows from the inequality $\dim H^r(G;M) \leq \dim H^r(G) \cdot \dim M$ for a $p$-group $G$ and an $\mathbb F_pG$-module $M$. This inequality is proved by considering a composition series for $M$, where every factor is a trivial module $\mathbb F_p$.

The second part is the statement that if $|G|=p^n$ then $\dim H^i (G) \leq \binom{n+i-1}{i}$. This can easily be proved by induction on $|G|$ by taking a central subgroup of order $p$ and considering the spectral sequence again.
\end{proof}

Now we concentrate on part (a) of condition (2).

If $H \leq G$ then by a collection of $G$-parameters for $H^*(H)$ we mean a finite collection of homogeneous elements of $H^*(G)$ that restricts to a collection of parameters on $H^*(H)$.

\begin{lemma}
\label{le:combine}
If $H$ is a normal subgroup of $G$ and we have a collection of parameters for $H^*(G/H)$ and a collection of $G$-parameters for $H^*(H)$, then the inflations of the former together with the latter form a collection of parameters for $H^*(G)$. 
\end{lemma}

\begin{proof}
We can check whether we have a collection of parameters for $H^*(G)$ by checking that the restrictions to any elementary abelian subgroup  $E$ form a system of parameters. This is because, by Quillen's $F$-isomorphism Theorem (see e.g.\ \cite[5.6.4]{Benson2}), it is enough to show that these elements restrict to a collection of parameters for $\lim_{\mathcal C_E} H^*(E)$, where the objects are the elementary abelian subgroups of $G$ and the morphisms are generated by inclusions and conjugations. But this limit can be realized as a submodule of $\oplus_E H^*(E)$, which will be finitely generated over the parameters. Consider the diagram below and apply cohomology.

\[
\xymatrix{
H \ar[r]  & G \ar[r]  & G/H \\
E \cap H \ar[r] \ar[u] & E \ar[r] \ar[u] & E/(E \cap H) \ar[u]
}
\]

By restriction we obtain a collection of parameters for $H^*(E/(E \cap H))$ and a collection of $E$-parameters for $H^*(E \cap H)$. The bottom row is split, so the inflations of the former together the latter combine to form a collection of parameters for  $H^*(E)$. 
\end{proof}

\section{Powerful, $\Omega$-Extendible Groups}
\label{se:first}

First we recall some standard definitions; in all of them $G$ is a finite $p$-group. The minimal number of generators of $G$ is denoted by $d(G)$ and the rank of $G$ is the maximum of the $d(H)$ as $H$ runs through the subgroups of $G$. The subgroup of $G$ generated by all the elements of order at most $p^r$ is denoted by $\Omega_r(G)$ and the subgroup generated by all the $p^r$th powers of elements of $G$ is denoted by $G^{p^r}$. The Frattini subgroup is $\Phi(G)=[G,G]G^p$ and the terms of the lower central series are denoted by $\gamma_i(G)$.

The group $G$ is said to be $p$-central if $\Omega_1(G)$ is contained in the centre and is called powerful if $[G,G] \leq G^p$ for odd $p$ or $[G,G] \leq G^4$ if $p=2$. 

There is a cohomological interpretation of the property of being powerful.

\begin{proposition}
	\label{pr:powerchar}
A $p$-group $G$ is powerful if and only if the only relations in $H^2(G)$ are squares.
\end{proposition}

Here we are implicitly taking a minimal set of homogeneous generators as a graded commutative ring. Of course, when $p$ is odd there are no relations in degree 2, but this formulation is valid for all $p$, odd or even.

\begin{proof} We have just rephrased \cite[2.7 and 2.10]{ms}.
\end{proof}

\begin{definition}(\cite[after 3.4]{ms}, \cite[\S1]{weigel})
A $p$-group $G$ is called $\Omega$-extendible if it is $p$-central and there exists a central extension $1 \rightarrow E \rightarrow F \rightarrow G \rightarrow 1$ such that $E=\Omega_1(F) =\Omega_2(F)^p$. The group is said to have the $\Omega$-extension property, or $\Omega$EP, if it is $p$-central and there is a $p$-central group $H$ such that $G \cong H/\Omega_1(H)$.
\end{definition}

It is easy to see that if $G$ is $\Omega$-extendible then it has the $\Omega$-extension property. In fact the two are equivalent when $p$ is odd: see Section~\ref{se:misc} for a discussion of these definitions.
Much of the literature refers to the $\Omega$-extension property, but $\Omega$-extendibility seems to work better at the prime 2.

There is also a cohomological characterisation of $\Omega$-extendibility. For an elementary abelian $p$-group $A$, let $B(A) \leq H^2(A)$ be the subgroup consisting of the image of the Bockstein map; this is the same as the image of $H^2(A; \mathbb Z)$. Write $A = \times_i C_i$ with each $C_i$ cyclic of order $p$ and let $x_i$ be a generator for $H^2(C_i)$. Then the inflations of the $x_i$ form a basis for $B(A)$.

\begin{proposition}
	\label{pr:omegachar}
	Let $G$ be a $p$-group and $A$ one of its maximal elementary abelian subgroups. Then 
		$G$ is $\Omega$-extendible if and only if every element of $B(A)$ is the restriction of an element of $H^2(G)$.
\end{proposition}

The proof is deferred to Section~\ref{se:misc}.

The following characterisation of groups that have cohomology ring isomorphic to that of some finite abelian group will be crucial. 

\begin{theorem}
\label{th:cohomchar}
For a finite $p$-group $G$, the following conditions are equivalent.
\begin{enumerate}
\item
$G$ is powerful and $\Omega$-extendible.
\item
There is a finite abelian $p$-group $A$ with $H^*(A) \cong H^*(G)$. More specifically, $H^*(G) \cong \Lambda[x_1, \ldots , x_m] \otimes \mathbb F_p[y_1, \cdots , y_n]$, where the $x_i$ have degree 1, the $y_i$ have degree at most 2 and $n=d(G)$. If $p$ is odd then $m=n$ and all the $y_i$ have degree 2; if $p=2$ then $m$ of the $y_i$ have degree 2, the rest having degree 1.
\item
(for $p$ odd only) $H^*(G)$ has a collection of parameters in degree 2 and the only relations in degree 2 are squares.
\end{enumerate}
\end{theorem}

\begin{remark} It is possible that (3) is equivalent to the other conditions even when $p=2$. It has been verified for groups of order at most 64 using the calculations of Green and King \cite{gk}.
	\end{remark}

\begin{proof} The equivalence of (1) and (2) is shown in \cite[3.14]{ms}  for powerful $G$; see also \cite{weigel} for $p\geq 3$ and \cite{bp} for $p \geq 5$. On the other hand, any $p$-group satisfying (2) or (3) is powerful, by Proposition~\ref{pr:powerchar}.
	
Clearly (2) implies (3). We do not use (3), so we postpone the remainder of the proof to Section~\ref{se:misc}.
\end{proof}

\section{Subgroups}

The next theorem was conjectured by D\'{i}az, Garaialde and Gonz\'{a}lez in \cite[5.2]{dgg2}, though without explicit bound.

\begin{theorem}
	\label{th:sub}
	If $G$ is any $p$-group with the property that any characteristic subgroup can be generated by at most $r$ elements (e.g.\ if $G$ is of rank at most $r$) then $G$ has a characteristic subgroup $N$ of rank at most $r$ that is powerful and $\Omega$-extendible and of index at most $p^{r (\lceil \log_2r \rceil +2 +e)}$ (the logarithm is rounded up to an integer), where $e=0$ for $p$ odd and $e=1$ for $p=2$.
	\end{theorem}

Without the $\Omega$-extendible condition, a version of this result was proved by Lubotzky and Mann \cite{LM}, based on work of B.W.~King. See also the treatments of Dixon et al.\ \cite{DDMS} and Khukhro \cite{khukhro}.

\begin{proof} First we present a proof for odd primes $p$, then we show how it can be adapted when $p=2$.
	
Let $V <G$ be the intersection of the kernels of all the homomorphisms from $G$ to $\Gl_r(\mathbb F_p)$.  It is shown in \cite[1.14]{LM} and \cite[2.13]{DDMS} that $V$ is powerful and of index bounded by $p^{r \lceil \log_2r \rceil}$. By hypothesis, $V$ is generated by at most $r$ elements so, being powerful, it has rank at most $r$ \cite[11.18]{khukhro}.

Set $H=V^p$; then $H$ is powerful \cite[11.6]{khukhro} and has index at most $p^r$ in $V$, since $V/H$ is elementary abelian \cite[11.10]{khukhro}. Now we show that $H$ is $p$-central; we do this by proving by induction on $i$ that $\Omega_1(H/H^{p^i})$ is invariant under conjugation by $H$.
Clearly this is true for $i=0$; assume that, for some $i$, $\Omega_1(H/H^{p^i})$ is invariant under conjugation by $H$, so it is elementary abelian of rank at most $r$. The kernel of the natural homomorphism $f \! :\Omega_1(H/H^{p^{i+1}}) \rightarrow \Omega_1(H/H^{p^i})$ is $H^{p^i}/H^{p^{i+1}}$, which is elementary abelian of rank at most $r$ since $H$ is powerful \cite[11.10]{khukhro}. The action of $V$ by conjugation on both $\Ker(f)$ and $\Ima(f)$ is trivial, by definition of $V$. Thus all that conjugation by $v \in V$ can do to $x \in \Omega_1(H/H^{p^{i+1}})$ is to send it to $xy$, where $y \in H^{p^i}/H^{p^{i+1}}$ so $y$ is centralised by $H$. It follows that conjugation by $v^p$ fixes $x$, hence $\Omega_1(H/H^{p^{i+1}})$ is invariant under conjugation by $H$. 

Set $N=H^p$; then $N$ is powerful \cite[11.6]{khukhro} and $p$-central and has index at most $p^r$ in $H$, since $H/N$ is elementary abelian. We will prove that $N$ is $\Omega$-extendible by showing that $B(\Omega_1(N))$ is in the image of restriction from $H^2(N)$ and invoking Proposition~\ref{pr:omegachar}. 

Consider the Lyndon-Hochschild-Serre spectral sequences for the inclusion $\Omega_1(N) \unlhd N$ 
with both $\mathbb F_p$ and $\mathbb Z$ coefficients and also for $\Omega_1(H) \unlhd H$
with $\mathbb F_p$ coefficients. We denote their terms by $E_*^{*,*}(N)$, $E_*^{*,*}(N; \mathbb Z)$ and $E_*^{*,*}(H)$. There are canonical comparison maps $\theta \! : E_*^{*,*}(N;\mathbb Z) \rightarrow E_*^{*,*}(N)$ and $\phi \! : E_*^{*,*}(H) \rightarrow E_*^{*,*}(N)$. We need to show that $B(\Omega_1(N)) \subseteq H^2(\Omega_1(N))=E^{0,2}_2(N)$ survives to $E^{0,2}_\infty(N)$.

We have $d_2(B(\Omega_1 (N))) = d_2\theta (E_2^{0,2}(N;\mathbb Z))=\theta d_2 (E_2^{0,2}(N;\mathbb Z))$. But $d_2 (E_2^{0,2}(N;\mathbb Z)) \subseteq E_2^{2,1}(N;\mathbb Z) \cong H^2(N/\Omega_1(N);H^1(\Omega_1(N); \mathbb Z)) =0$, so $B(\Omega_1 (N))$ survives to $E_3^{0,2}(N)$. Likewise, $B(\Omega_1 (H))$ survives to $E_3^{0,2}(H)$. The inclusion $\Omega_1(N) \leq \Omega_1(H)$ is split, so  $\phi$ maps $B(\Omega_1 (H))$ onto $B(\Omega_1 (N))$ and $d_3 (B(\Omega_1(N))) = d_3 \phi (B(\Omega_1 (H))) = \phi d_3 (B(\Omega_1 (H))) \subseteq \phi (E_3^{3,0}(H))$.

Notice that $H/\Omega_1(H)$ satisfies the definition of the $\Omega$-extension property, so by Lemma~\ref{le:omega} it is $\Omega$-extendible and as a quotient of a powerful group it is powerful. By Theorem~\ref{th:cohomchar} its cohomology is generated in degrees 1 and 2, so $H^3(H/\Omega_1(H)) = H^1(H/\Omega_1(H)) \cdot H^2(H/\Omega_1(H))$. But $\phi H^1(H/\Omega_1(H)) =0$, since $H^1(-) \cong \Hom(-, \mathbb F_p)$, thus $\phi H^3(H/\Omega_1(H)) =0$ and hence $\phi E_3^{3,0}(H)=0$. All subsequent differentials are clearly 0 on $E_*^{0,2}(H)$, so $B(\Omega_1 (N))$ survives.

For $p=2$ the proof is similar and we just note the differences. We replace $V$, $H$ and $N$ by $V'=V^2$, $H'=H^2$ and $N'=N^2$. It is shown in \cite[4.1.14]{LM} and \cite[2.13]{DDMS} that $V'$ is powerful and of index bounded $p^{r( \lceil \log_2r \rceil+1)}$. By the same method as before, we show that $\Omega_2(H')$ is central in $H'$ and thus abelian, hence $H'/\Omega_1(H')$ is $2$-central. Write $\Omega_2(H') \cong (\mathbb Z/4)^a \times (\mathbb Z/2)^b$ and let $X= (\mathbb Z/2)^b$. The extension $1 \rightarrow (\mathbb Z/2)^a \rightarrow H'/X \rightarrow H'/\Omega_1(H')) \rightarrow 1$ shows that $H'/\Omega_1(H')$ is $\Omega$-extendible. Note that $(\mathbb Z/2)^a = \Omega_1(H'/X)$, hence $\Omega_2(H'/X)=\Omega_2(H')/X$. The rest of the proof is the same.
\end{proof}

\section{Parameters}
\label{sec:para}

We are going to use Proposition~\ref{pr:par} with $\{ G_{\lambda} \}$ the set of $p$ groups of rank $r$ up to isomorphism. Lemma~\ref{le:dimbound} shows that the second condition is satisfied by taking $H_\lambda$ to be the subgroup of $G_\lambda$ from Theorem~\ref{th:sub}. Theorem~\ref{th:cohomchar} shows that we can take the number $U(i)$ to be $\binom{i+r-1}{i}$.

For the first condition we use Lemma~\ref{le:combine} to produce a collection of parameters. There are only finitely many possibilities for the isomorphism class of $G_\lambda/H_\lambda$, since its order is bounded, so we choose a collection of parameters for each one. Altogether this is a finite set, so there is some bound on the number of parameters and on their degrees. In fact, these bounds can be made explicit. One way of producing parameters is to find a faithful representation of the group and take its Chern classes; the proof of \cite[10.2]{Sy1} shows that this yields a crude bound of $|G_\lambda/H_\lambda|$ on the number of parameters and $2|G_\lambda/H_\lambda|$ on their degrees. The remaining problem is to find a collection of $G_\lambda$-parameters for $H_\lambda$ with the degrees bounded independently of $\lambda$.

Reverting to the notation of the previous section, $G$ has a characteristic subgroup $N$ of bounded index that is powerful and $\Omega$-extendible. Thus $\Omega_1(N)$ is elementary abelian and there are elements $x_1, \ldots , x_n \in H^2(N)$ such that $\res_{\Omega_1(N)}x_1, \ldots , \res_{\Omega_1(N)}x_n$ form a basis for $B(\Omega_1(N))$. Note that $n$ is at most the rank of $G$.

Let $c_{n,0}, \ldots , c_{n,n-1}$ be the Dickson polynomials over $\mathbb F_p$ in the $x_i$ (see e.g.\ \cite[8.1]{Benson}) and set $d_i= \mathcal N ^G_N (c_{n,i-1})$, where $\mathcal N$ denotes the Evens norm map. We claim that $d_1, \ldots , d_n$ is a collection of $G$-parameters for $H^*(N)$; it is sufficient to check that it forms a collection of $G$-parameters for $H^*(\Omega_1(N))$. On applying the Mackey formula to $\res^G_{\Omega_1(N)} \mathcal N ^G_N (c_{n,i-1})$ we obtain a product of conjugates of  $\res^N_{\Omega_1(N)} c_{n,i-1}$ by elements of $G$. But the $\res ^N_{\Omega_1(N)}c_{n,i-1}$ are naturally the images of the Dickson polynomials in the polynomial ring $\mathbb F_p [ B(\Omega_1(N))]$, on which $G$ acts via conjugation. The latter Dickson polynomials are certainly invariant under $G$ and form a system of parameters. Thus $\res^G_{\Omega_1(N)} (d_i)$ is a power of $c_{n,i-1}$ and so the $\res^G_{\Omega_1(N)} (d_i)$ form a system of parameters for $H^*(\Omega_1(N))$.

The $d_i$ are bounded in degree by $2(p^n-1)[G:N]$. This completes the proof of Theorem~\ref{th:main}.

\section{Miscellaneous Results}
\label{se:misc}

This section contains the proofs of sundry results that were relegated here so as not to encumber the main exposition.

\begin{proposition}
	\label{pr:rank}
	There is a function $F(p,r)$ such that any $p$-group of coclass $r$ has rank at most $F(p,r)$.
	\end{proposition}

This result appears to be known to coclass aficionados, but we are not aware of a proof in the literature.
 
 \begin{proof}
 	We sketch the proof for odd $p$; the proof for $p=2$ is similar but the theorems cited change accordingly. We employ Leedham-Green's Structure Theorem for $p$-groups \cite[11.3.9]{LG}. This states that, with finitely many exceptions (which does not matter for this proof), any $p$-group of coclass $r$ possesses a normal subgroup $N$ of order bounded in terms of $p$ and $r$ such that $P/N$ is constructible. The definition of constructible involves many things and we summarise the relevant ingredients here \cite[8.4.3, 8.4.9]{LG}. 
 	
 	There are normal subgroups $L \leq M$ of $P/N$ with $L$ and $M/L$ abelian and we set $Q=(P/N)/M$. There is a uniserial $\hat{\mathbb Z}_pQ$-lattice $T$ with $\hat{\mathbb Z}_pQ$-sublattices $U \leq V \leq T$ such that $L \cong V/U$ and $M/L \cong T/V$ as $\hat{\mathbb Z}_pQ$-modules. There is an extension $1 \rightarrow T \rightarrow R \rightarrow Q \rightarrow 1$ compatible with the action of $Q$ on $T$ and $R$ has coclass at most $r$.
 	
 	By \cite[7.4.13]{LG}, the rank of $T$ as a lattice is bounded in terms of $p$ and $r$. From the first paragraph of the proof of \cite[11.3.9]{LG} we see that, in the all but finitely many cases considerd there, $|Q|$ is bounded in terms of $p$ and $r$. We have bounded the ranks of $N$, $L$, $M/L$ and $P/M$; rank is subadditive on short exact sequences, so we are done.
 	\end{proof}

 \begin{proof}[Proof of Theorem~\ref{th:gt}.]
 	Take the normal subgroup $N$ from Theorem~\ref{th:sub} and estimate the dimensions of the entries on the $E_2$-term of the Lyndon-Hochschild-Serre spectral sequence, as in the proof of Lemma~\ref{le:dimbound}. From Theorem~\ref{th:cohomchar} we obtain $\dim H^j(N) \leq \binom{j+r-1}{j}$ and the last part of Lemma~\ref{le:dimbound} shows that $\dim H^k(G/N) \leq \binom{r(\lceil \log _2 r \rceil +2+e) +k-1}{k}$. Now sum along the diagonals, This is easier to do if we observe that we are bounding the respective Hilbert series by the Hilbert series $(1-t)^{-r}$ and $(1-t)^{-r(\lceil \log _2 r \rceil +2+e)}$ then multiplying the Hilbert series.
 	\end{proof}
 
 \begin{proof}[Proof of Theorem~\ref{thm:gt2}.]
 Take the collection of parameters for $H^*(G)$ produced in Section~\ref{sec:para}; call them $d_1, \ldots , d_m$. It is bounded in number and its elements are bounded in degree by functions of $p$ and $r$. We can take suitable powers $e_i=d_i^{\alpha_i}$ of these parameters so that they all lie in the same degree $D$. The product of the degrees of all the parameters is a possible value for $D$, and it is bounded by a function of $p$ and $r$. We can now take a Noether normalization of $H^*(G)$ as a $\mathbb F_p[e_1, \ldots , e_m]$-module in such a way that the new parameters $f_1, \ldots , f_a$ are linear combinations of the old ones \cite[13.3]{eisenbud}. This might require a field extension $\mathbb F_p <k$, but that will not affect $\dim_k H^i(G;k)$. We know that there are precisely $a$ new parameters because $a$ is the Krull dimension of $H^*(G)$, by Quillen \cite[7.8]{quillen}. Because $\reg H^*(G)=0$ \cite[0.2]{Sy1}, we see from \cite[2.1]{Sy2} that $H^*(G)$ is generated as a module over $k[f_1, \ldots ,f_a]$ in degrees at most $\sum_i (\deg (f_i)-1)=a(D-1)$.
 
 Finally, Theorem~\ref{th:gt} shows that $\dim \oplus_{i=0}^{i=a(D-1)} H^i(G)$ is bounded by a function of $p$ and $r$. The theorem follows.
 	\end{proof}
 
 All the bounds in this proof can easily be written down, yielding an explicit but unwieldy formula for $X(p,r)$.
 
 If we had a result that bounded the regularity of the cohomology of a virtual Poincar\'{e} duality pro-$p$ group then Corollary~\ref{cor:pro-p} could be proved in the same way as Theorem~\ref{th:main}.

 \begin{proof}[Proof of Corollary~\ref{cor:pro-p}.]
 Let $G$ be a pro-$p$ group with rank bounded by $r$. By \cite[4.3]{DDMS}, $G$ has an open normal subgroup $U$ that is uniformly powerful. By a result of Lazard (cf. \cite[3.11]{ms}), $H^*(U)$ is the exterior algebra on $H^1(U)$, so has bounded dimension. A spectral sequence argument now shows that $H^*(G)$ is a noetherian ring \cite[13.5]{quillen}. Recall that $H^i(G)$ is the direct limit of the $H^i(G/N)$ as $N$ runs through the open normal subgroups of $G$ and observe that a bound on the rank of $G$ is inherited by the $G/N$.
 
 Since $H^*(G)$ is noetherian, there is an $N_j$ such that inflation $\inf^{G/N_j}_G \! : H^*(G/N_j)  \rightarrow H^*(G)$ is onto; the pro-$p$ version of Theorem~\ref{th:gt} thus follows from the original one. There must also be an $N_k \leq N_j$ such that $\inf_{G/N_k}^G \! : \Ima (\inf^{G/N_j}_{G/N_k} \! :H^*(G/N_j) \rightarrow H^*(G/N_k)) \rightarrow H^*(G)$ is an isomorphism. There are only finitely many possibilities for $H^*(G/N_j)$ and $H^*(G/N_k)$ as rings, by Theorem~\ref{th:main}; the inflation maps are determined by the images of the generators, so there are only finitely many of them, hence finitely many images.
 \end{proof}
	
\begin{lemma}
	\label{le:pcent}
	For $p$ odd, the condition that $G$ be $p$-central in the definition of the $\Omega$-extension property is redundant. In other words, it is a consequence of the other condition.
	\end{lemma}

\begin{proof}
	This is shown at the beginning of the proof of \cite[III.12.2]{huppert}.
	\end{proof}

\begin{example}
If $Q$ is a generalised quaternion 2-group then $Q/\Omega_1(Q)$ is a dihedral group, which is not $p$-central if it has order greater than 4.
\end{example}
		
\begin{lemma}
	\label{le:omega}
	For $p$ odd, a $p$-group has the $\Omega$-extension property if and only if it is $\Omega$-extendible.
	\end{lemma}

When $p=2$ the status of this lemma is not clear.

\begin{proof}
	Clearly if $G$ is $\Omega$-extendible then it has the $\Omega$-extension property. For the converse, let $H$ be a $p$-central group such that $G = H/\Omega_1(H)$.
Either by definition or by the previous lemma, $G$ is $p$-central, so $\Omega_2(H)/\Omega_1(H)$ is elementary abelian. Since $p$ is odd, the $p$-power map  $\Omega_2(H)/\Omega_1(H) \rightarrow \Omega_1(H)$ is a homomorphism, by \cite[6.14]{khukhro}; its kernel is clearly trivial. If this map is not onto, write $\Omega_1(H)= \Omega_2(H)^p \times X$ and replace $H$ by $H/X$ to obtain a map that is onto.
	\end{proof} 

\begin{proof}[Proof of Proposition~\ref{pr:omegachar}.] Suppose that every element of $B(A)$ is the restriction of an element of $H^2(G)$. Consider the action of $N_G(A)$ on $H^2(A)$ via conjugation. This must fix $B(A)$ since the latter is generated by restrictions. But $B(A)$ is canonically isomorphic to $H^1(A)$ by the Bockstein map, so $N_G(A)$ acts trivially on $H^1(A)$ and hence centralises $A$. A straightforward argument as in \cite[2.3]{weigel} now shows that $A$ is central and hence $G$ is $p$-central. Let $n$ denote the rank of $A=\Omega_1(G)$.

Let $z_1, \ldots , z_n \in H^2(G)$ be elements that restrict to a basis of $B(A)$. They combine to define an element $\underbar{z} \in H^2(G;(\mathbb Z/p)^n)$ and hence a central extension of groups $1 \rightarrow (\mathbb Z/p)^n \rightarrow F \rightarrow G \rightarrow 1$. Because the $z_i$ restrict to a basis of $B(A)$, we have $\Omega_2(F) \cong (\mathbb Z/p^2)^n$, so this extension certainly satisfies the definition of $\Omega$-extendible.

Conversely, suppose that $G$ satisfies the definition of $\Omega$-extendible, so in particular $A=\Omega_1(G)$. After choosing an isomorphism $E \cong (\mathbb Z/p)^n$ we obtain elements $z_1, \ldots z_n \in H^2(G)$ that describe the extension in the definition. For the moment, assume that $G$ cannot be expressed non-trivially as a product $K \times X$ with $X$ elementary abelian; since $G$ is $p$-central, this is equivalent to $\Omega_1(G) \leq \Phi(G)$. By \cite[1.3]{ms}, the image of restriction $H^2(G) \rightarrow H^2(\Omega_1(G))$ lies in $B(\Omega_1(G))$. Thus $\Omega_2(F)$ is abelian and it is easy to check that the only possibility that satisfies the property $\Omega_2(F)^2=\Omega_1(F)$ is when $\Omega_2(F) \cong (\mathbb Z/p^2)^n$ and thus the elements $\res^G_{\Omega_1(G)}z_i$ form a basis for $B(\Omega_1(G))$, as required. 

If $G= K \times X$, we may assume that $K$ cannot be decomposed further. Let $F'$ be the inverse image of $K$ in $F$ and write $E=(F')^p \times Y$. The extension $1 \rightarrow E/Y \rightarrow F'/Y \rightarrow K \rightarrow 1$ shows that $K$ satisfies the definition of $\Omega$-extendible and so the discussion above produces elements $z'_1, \ldots , z'_m \in H^2(K)$, which can be inflated to $z_1, \ldots , z_m \in H^2(G)$. We obtain $z_{m+1}, \ldots , z_n$  by taking a basis for $B(X)$ and inflating. 
\end{proof}

 Notice that this proof shows that, in the definition of $\Omega$-extendible, we can always assume that $\Omega_2(F) \cong (\mathbb Z/p^2)^n$.
 
\begin{proof}[Proof of Theorem~\ref{th:cohomchar}.] All that remains is to show that if $G$ satisfies condition (3) then it is $\Omega$-extendible. Let $A \leq G$ be a maximal elementary abelian subgroup. Since $H^2(G)$ contains a collection of parameters for $H^*(G)$, its restriction to $H^2(A)$ must contain a collection of parameters for $H^*(A)$. Since $p$ is odd, this means that the restrictions of the parameters span $H^2(A)/(H^1(A))^2$, so the action of $N_G(A)$ on this by conjugation must be trivial. But $H^2(A)/(H^1(A))^2$ is canonically isomorphic via the Bockstein to $H^1(A)$, so the same argument as in the previous proof shows that $G$ is $p$-central.

Let $z_1, \ldots , z_n \in H^2(G)$ be elements that restrict to a basis for $H^2(\Omega_1(G))/(H^1(\Omega_1(G)))^2$. These combine to define an extension that shows that $G$ is indeed $\Omega$-extendible, by \cite[2.2]{ms}. 
\end{proof}

The last sentence of the proof of Theorem~5.1 (Carlson's Conjecture for $p=2$) in \cite{carlson} is unclear in its use of Theorem~3.3 of that work. Here we present a proof of an amended version of Theorem~3.3  that can be used instead. It is a version of \cite[4.2]{dgg2} for the prime 2.

\begin{theorem} Let $f$, $n$ and $r$ be non-negative integers and suppose that 
	\[ 1 \rightarrow N \rightarrow G \rightarrow Q \rightarrow 1 \]
	is an extension of finite $p$-groups such that
	\begin{enumerate}
		\item
		$|N| \leq n$ and
		\item
		$Q$ has an abelian normal subgroup $A$ with $|Q:A| \leq f$ and $\rank (A) \leq r$.
		\end{enumerate}
	Then the ring $H^*(G)$ is determined up to a finite number of possibilities by $H^*(Q)$.
	\end{theorem}

\begin{proof}
	We use induction on $n$ (for all $G$, $f$, $r$). The statement is trivially true for $n=0$; otherwise pick some $\tilde{N} <N$ of index $p$ that is normal in $G$ and set $\tilde{G}=G/\tilde{N}$. Consider the induced extension 
	\[ 1 \rightarrow C_p \rightarrow \tilde{G} \xrightarrow{\pi} Q \rightarrow 1 \]
	and set $\tilde{A} = \pi^{-1}(A)$, $B=\pi^{-1}(A^p)$. From the standard formula $[ab,c]=[a,c]^b[b,c]$ \cite[1.11]{khukhro} we readily obtain $[a^p,b]=[a,b]^p=1$ for $a,b \in \tilde{A}$. Since any element of $B$ is of the form $a^pc$, $a \in \tilde{A}$, $c \in C_p$, it follows that $B$ is abelian of rank at most $r+1$. 
	
	Now $| \tilde{A} : B | = |A:A^p| \leq p^r$ and $|G:\tilde{A}| = |Q:A| \leq f$, so $|\tilde{G}:B| \leq p^rf$. Apply \cite[3.1, 3.2]{carlson} to see that there are only finitely many possibilities for $H^*(\tilde{G})$. Consider the extension
	\[ 1 \rightarrow \tilde{N} \rightarrow G \rightarrow \tilde{G} \rightarrow 1 \]
	and apply the induction hypothesis, with $B$ as normal subgroup, to deduce that there are only finitely many possibilities for $H^*(G)$.
	\end{proof}


\end{document}